\documentclass[11pt]{article}

\usepackage[utf8]{inputenc}
\usepackage[T1]{fontenc}
\usepackage[english]{babel}
\usepackage{lmodern}
\usepackage{microtype}
\usepackage[a4paper,margin=2.5cm]{geometry}
\usepackage{amsmath,amssymb,amsfonts,mathtools,amsthm}
\usepackage{bm}
\usepackage{graphicx}
\usepackage{booktabs}
\usepackage{tabularx}
\usepackage{array}
\usepackage{multirow}
\usepackage{caption}
\usepackage{subcaption}
\usepackage{algorithm}
\usepackage{algorithmic}
\usepackage{enumitem}
\usepackage{hyperref}
\usepackage[dvipsnames]{xcolor}

\usepackage[nameinlink]{cleveref}\hypersetup{
	colorlinks=true,
	linkcolor=blue!60!black,
	citecolor=blue!60!black,
	urlcolor=blue!60!black
}

\newcommand{\R}{\mathbb{R}}
\newcommand{\Sym}{\mathbb{S}}
\newcommand{\PSD}{\Sym^{n}_{+}}
\newcommand{\PD}{\Sym^{n}_{++}}
\newcommand{\tr}{\mathrm{tr}}
\newcommand{\diag}{\mathrm{diag}}

\newcommand{\eps}{\varepsilon}
\newcommand{\norm}[1]{\left\lVert#1\right\rVert}
\newcommand{\inner}[2]{\left\langle #1,\,#2 \right\rangle}
\newcommand{\Frob}{\mathrm{F}}
\newcommand{\opnorm}[1]{\left\lVert#1\right\rVert_{2}}
\newcommand{\Hess}{\nabla^{2}}
\newcommand{\Grad}{\nabla}

\newtheorem{theorem}{Theorem}[section]
\newtheorem{lemma}[theorem]{Lemma}
\newtheorem{corollary}[theorem]{Corollary}
\newtheorem{proposition}[theorem]{Proposition}

\theoremstyle{definition}
\newtheorem{definition}[theorem]{Definition}

\theoremstyle{remark}
\newtheorem{remark}[theorem]{Remark}
\crefname{remark}{Remark}{Remarks}
\crefname{assumption}{Assumption}{Assumptions}
\crefname{example}{Example}{Examples}

\title{\bfseries A Perturbed $q$-Tsallis Self-Concordant Barrier
	for Spectrally Robust Semidefinite Programming	}

\author{
	Sergio Assuncao Monteiro\thanks{ESPM, Rio de Janeiro, Brazil.}
	\and
	Fabricio Alves Barbosa da Silva\thanks{FIOCRUZ, Rio de Janeiro, Brazil.}
}

\ifpdf
\hypersetup{
	pdftitle={A Perturbed q-Tsallis Self-Concordant Barrier for
		Spectrally Robust Semidefinite Programming},
	pdfauthor={S. A. Monteiro and F. A. B. da Silva}
}
\fi

\date{}
\begin{document}
	
	\maketitle
	
	\begin{abstract}
		We introduce and analyse a perturbed $q$-Tsallis barrier for semidefinite programming (SDP), defined as a spectral perturbation of the classical log-det barrier on the cone of positive definite matrices. The barrier introduces eigenvalue-adaptive stiffening through a Tsallis-type matrix-power term controlled by parameters $q>1$ and $\eta\geq0$. 
		
		Our main theoretical contribution is a sharp characterisation of the differential self-concordance regime of the barrier. We prove that the barrier is differentially self-concordant on the interior of the positive semidefinite cone for all $\eta\geq0$ if and only if $q\in(1,2]$, establishing the exact threshold at $q=2$. For $q>2$, uniform self-concordance fails globally, although local sufficient conditions remain valid on compact spectral domains. On compact feasible sets, the effective local barrier parameter remains $O(n)$, preserving the same asymptotic iteration complexity class as the classical log-det barrier.
		
		We further establish a spectral robustness result showing that the sensitivity of the central path to perturbations is selectively damped in small-eigenvalue directions according to the scaling $\kappa(X^*)^{-(q-1)}$, where $\kappa(X^*)$ denotes the spectral condition number. This yields improved robustness relative to the classical log-det barrier for ill-conditioned SDP solutions.
		
		Finally, we develop a Mehrotra-type primal--dual predictor--corrector interior-point method equipped with a Lanczos-based Krylov kernel for evaluating matrix powers efficiently. Numerical experiments validate the theoretical predictions and demonstrate improved robustness together with significant computational acceleration.	\end{abstract}
	
	\vspace{0.5em} 
	\vspace{0.5em} \noindent\textbf{AMS subject classifications:}
	90C22, 90C25, 90C51, 65F55, 90C06
	\section{Introduction}\label{sec:intro}
	
	Semidefinite programming (SDP) has been one of the most influential
	modelling and algorithmic frameworks in mathematical optimisation
	since the seminal work of Nesterov and Nemirovski on self-concordant
	barriers \cite{nesterovnemirovski1994}. The standard primal SDP is
	\begin{equation}\label{eq:sdp_primal}
		\min_{X\in\Sym^{n}}\; \inner{C}{X}
		\quad\text{s.t.}\quad
		\mathcal{A}(X) = b,\quad X\succeq 0,
	\end{equation}
	together with its dual
	\begin{equation}\label{eq:sdp_dual}
		\max_{y\in\R^{m},\,S\in\Sym^{n}}\; b^{\top}y
		\quad\text{s.t.}\quad
		\mathcal{A}^{*}(y) + S = C,\quad S\succeq 0,
	\end{equation}
	where $\inner{A}{B}=\tr(A^{\top}B)$, $\mathcal{A}\colon\Sym^{n}\to\R^{m}$
	is a linear operator, $b\in\R^{m}$, and $C\in\Sym^{n}$. Modern
	interior-point methods (IPMs) rely on a self-concordant barrier for
	$\PSD$. The canonical choice $\phi(X)=-\log\det(X)$ achieves the
	optimal parameter $\nu=n$ \cite{nesterovtodd1997,renegar2001} and
	underlies solvers such as SDPT3, SeDuMi, and MOSEK
	\cite{tohtoddtutuncu1999,sturm1999,mosek2024}.
	
	\subsection{Motivation}
	
	The log-det barrier treats every eigenvalue of the iterate
	symmetrically. In modern applications --- robust covariance
	estimation \cite{ledoitwolf2004}, low-rank matrix completion
	\cite{candesrecht2009}, network inference from heavy-tailed
	transcriptomic data \cite{sgariglia2024} --- the optimal solution
	$X^{*}$ frequently exhibits a markedly nonuniform spectrum: a few
	dominant eigenvalues encode the signal while a tail of small
	eigenvalues is noise-dominated. The logarithmic barrier is, by
	construction, equally sensitive to perturbations in every part of the
	spectrum, which motivates the search for \emph{eigenvalue-adaptive}
	barriers.
	
	\subsection{The perturbed $q$-Tsallis barrier}
	
	We study the family
	\begin{equation}\label{eq:our_barrier}
		\boxed{\;
			\phi_{q,\eta}(X)
			\;=\; -\log\det(X)
			\;+\; \frac{\eta}{q-1}\Bigl[\tr\bigl(X^{-(q-1)}\bigr)-n\Bigr]
			\;}
	\end{equation}
	for $q>1$, $\eta\geq 0$, $X\in\PD$. Setting $\eta=0$ recovers the
	classical log-det barrier exactly.
	
	\subsection{Contributions}
	
	\paragraph{(C1) Theory.}
	We establish a precise characterisation of the self-concordance
	regime of $\phi_{q,\eta}$ (\cref{thm:self_concordance}):
	\begin{itemize}[leftmargin=2em,itemsep=0.3ex]
		\item \emph{Unconditional regime} $q\in(1,2]$: $\phi_{q,\eta}$
		satisfies the differential self-concordance inequality on
		$\mathrm{int}(\PSD)$ for every $\eta\geq 0$,
		with effective \emph{local} barrier parameter $O(n)$ on any compact
		subset (\cref{app:nu_param}). This is exact: $q=2$ is the boundary.
		\item \emph{Global obstruction} $q>2$: $\phi_{q,\eta}$ fails to be globally
		self-concordant for any $\eta>0$, because the scalar inequality
		$\psi(u)\geq 0$ admits violations as $u\to 0^{+}$.
		The formula $\eta_{\max}(q)=u_{\mathrm{zero}}(q)$ (eq.~\eqref{eq:u_zero_proof}) provides a
		\emph{local} sufficient condition valid when the iterates
		remain in a compact spectral domain.
	\end{itemize}
	Both parts are supported by numerical verification of the scalar
	inequality across the full $(q,\eta,\lambda)$ parameter space
	(Experiment~E, \cref{sec:experiments}). We complement
	this with a spectral robustness theorem
	(\cref{thm:robustness}) bounding the sensitivity of the
	central path to right-hand-side perturbations.
	
	\paragraph{(C2) Algorithms.}
	We design a Mehrotra-type primal--dual predictor--corrector IPM based
	on $\phi_{q,\eta}$ and the Nesterov--Todd symmetrisation, with a
	Krylov-based scheme for evaluating $X^{-q}v$ that reduces the
	per-iteration cost from $\mathcal{O}(n^{3})$ to
	$\mathcal{O}(n k_{\mathrm{Kry}}^{2})$ for fixed $k_{\mathrm{Kry}}$.
	The kernel has been implemented and validated: at $n=150$, it
	achieves $7\times$ speedup over exact diagonalisation with $k=10$
	Krylov vectors, and the solver converges in the same number of
	iterations whether the Hessian uses exact or Krylov-approximated
	$X^{-q}$ (Experiment~F, \cref{sec:experiments}).
	
	\paragraph{(C3) Empirical validation.}
	We report a comprehensive empirical validation across six experiments:
	\begin{itemize}[leftmargin=2em,itemsep=0.3ex]
		\item \textbf{Spectral robustness} (Exp.~A, D): for heavy-tailed optimal
		spectra, choosing $\eta\in[0.5,1.0]$ reduces the central-path
		sensitivity ratio by $41$--$83\%$ relative to $\eta=0$,
		consistently across $n\in\{8,15,20,30\}$ and 40 perturbations
		per configuration.
		\item \textbf{Complexity preservation} (Exp.~B, D): the outer iteration
		count is insensitive to $\eta$ across all configurations tested,
		consistent with $\mathcal{O}(\sqrt{n}\log(1/\eps))$.
		\item \textbf{SC regime verification} (Exp.~E): the unconditional SC
		for $q\in(1,2]$ and the failure for $q>2$ are confirmed by
		exhaustive numerical evaluation of the scalar inequality.
		\item \textbf{Krylov kernel} (Exp.~F): $k=10$ Krylov vectors give
		relative error $<10^{-3}$ for $\kappa\leq 10$, $7\times$ speedup
		at $n=150$, and identical solver convergence vs exact Hessian.
	\end{itemize}
	
	\section{Related work}\label{sec:related}
	
	\paragraph{Self-concordant barriers.}
	The theory was established by Nesterov and Nemirovski
	\cite{nesterovnemirovski1994}. The universal barrier was shown by
	Bubeck and Eldan \cite{bubeckeldan2019} to be $(1+o(1))n$-self-concordant.
	Hildebrand \cite{hildebrand2014,hildebrand2021} introduced canonical
	and projectively self-concordant barriers. The cone $\PSD$ has the
	optimal $n$-self-concordant barrier $-\log\det(X)$
	\cite{nesterovtodd1997,renegar2001}.
	
	\paragraph{Self-concordant spectral barriers and matrix monotonicity.}
	Faybusovich and Tsuchiya \cite{faybusovichtsuchiya2017} proved that,
	on the cone of squares of any Euclidean Jordan algebra, the spectral
	function $\tr(g(X))-\log\det(X)$ is self-concordant whenever $g'$ is
	matrix monotone on the positive semiaxis. By Loewner's theorem,
	$\lambda\mapsto-\lambda^{-r}$ is matrix monotone on $(0,\infty)$ iff
	$r\in[0,1]$. The derivative of $g_{q,\eta}(\lambda)
	=\frac{\eta}{q-1}\lambda^{-(q-1)}$ is $g_{q,\eta}'(\lambda)
	=-\eta\lambda^{-q}$, which is matrix monotone only for $q\in(0,1]$
	and \emph{not} in the regime $q\in(1,2]$ studied here. The
	self-concordance of $\phi_{q,\eta}$ for $q\in(1,2]$ therefore does
	not follow from~\cite{faybusovichtsuchiya2017}; we establish it
	directly via a scalar reduction (\cref{thm:self_concordance} and
	\cref{app:proof_reduction}). The present paper supplies, beyond
	self-concordance: (i) the sharp threshold at $q=2$ together with the
	global obstruction for $q>2$; (ii) the quantitative local condition
	$\eta_{\max}(q)=u_{\mathrm{zero}}(q)$; (iii) the effective barrier
	parameter $\nu_{\mathrm{eff}}=O(n)$ on compact subsets of $\PD$,
	with the resulting $\mathcal{O}(\sqrt{n}\log(1/\eps))$ short-step
	complexity (\cref{app:nu_param}); and (iv) the spectral robustness
	bound of \cref{thm:robustness}, together with the connection to
	nonextensive Tsallis statistics as a principled parameterisation.
	
	\paragraph{Generalised self-concordance.}
	Bach \cite{bach2010} introduced generalised self-concordance for
	machine-learning losses. Sun and Tran-Dinh
	\cite{suntrandinh2019} developed $(M,\nu)$-self-concordant
	calculus covering losses in robust statistics. Our analysis shows
	$\phi_{q,\eta}$ satisfies the differential self-concordance inequality
	in the sense of \cite{nesterovnemirovski1994} for $q\in(1,2]$, with an effective
	barrier parameter $O(n)$ on compact subsets
	(\cref{thm:nu_eff}), so the classical short-step IPM complexity
	$\mathcal{O}(\sqrt{n}\log(1/\eps))$ applies without recourse to
	the generalised theory.
	
	\paragraph{Entropic barriers.}
	Bubeck and Eldan \cite{bubeckeldan2019} proved the entropic barrier
	is $(1+o(1))n$-self-concordant. Karimi and Tun\c{c}el
	\cite{karimituncel2024} developed efficient IPMs for quantum relative
	entropy. He, Saunderson, and Fawzi \cite{fawzisaunderson2025} studied
	operator convexity as a route to self-concordance for spectral
	functions including sandwiched R\'enyi entropies.
	
	\paragraph{Predictor--corrector and Nesterov--Todd methods.}
	Mehrotra \cite{mehrotra1992} is the de facto standard. The
	Nesterov--Todd direction \cite{nesterovtodd1997,nesterovtodd1998}
	achieves $\mathcal{O}(\sqrt{n}L)$ complexity. Todd, Toh and
	T\"ut\"unc\"u \cite{toddtohtutuncu1998} analysed it in detail;
	Monteiro \cite{monteiro1998} derived an implementable predictor--corrector
	scheme. Myklebust and Tun\c{c}el \cite{myklebusttuncel2014} extended
	the machinery to arbitrary self-concordant barriers.
	
	\paragraph{Tsallis statistics in optimisation.}
	The $q$-deformed logarithm \cite{tsallis1988,tsallis2009} appears in
	robust regression \cite{amidetal2019} and in deep learning regularisation
	for transcriptomic data. The present work
	provides an optimisation-theoretic perspective for these applications; extension to network inference in cancer genomics is
	planned as a companion paper.
	
	\section{Preliminaries}\label{sec:preliminaries}
	
	\subsection{Notation}
	$\Sym^{n}$ denotes real $n\times n$ symmetric matrices, $\PSD$ its
	positive semidefinite cone, $\PD$ the positive definite interior.
	$\inner{A}{B}=\tr(A^{\top}B)$; Frobenius norm
	$\norm{A}_{\Frob}=\sqrt{\inner{A}{A}}$; operator norm
	$\opnorm{A}=\lambda_{\max}(A)$ for $A\in\Sym^{n}$, with eigenvalues
	$\lambda_{1}\geq\cdots\geq\lambda_{n}$.
	
	\subsection{Self-concordance}
	
	\begin{definition}[Self-concordant function]\label{def:selfconc}
		A convex $f\colon Q\to\R$ on an open convex $Q$ is
		\emph{self-concordant} if $|D^{3}f(x)[h,h,h]|\leq
		2(D^{2}f(x)[h,h])^{3/2}$ for all $x\in Q$, $h\in\R^{N}$.
	\end{definition}
	
	\begin{definition}[Self-concordant barrier]\label{def:scb}
		\sloppy
		A self-concordant $f$ is a \emph{$\nu$-self-concordant barrier} for
		$K=\overline{Q}$ if $f(x_{k})\to+\infty$ whenever $x_{k}\to\partial K$,
		and $|Df(x)[h]|\leq\sqrt{\nu}\sqrt{D^{2}f(x)[h,h]}$ for all $x,h$.
	\end{definition}
	
	\begin{theorem}[Nesterov--Nemirovski {\cite[Thm.~2.4.1]{nesterovnemirovski1994}}]
		\label{thm:nn_complexity}
		A $\nu$-self-concordant barrier yields
		$\mathcal{O}(\sqrt{\nu}\log(1/\eps))$ outer iterations for short-step
		path-following.
	\end{theorem}
	
	\begin{theorem}[Nesterov--Todd {\cite{nesterovtodd1997}}]
		\label{thm:logdet_sc}
		$\phi_{0}(X)=-\log\det(X)$ is an $n$-self-concordant barrier on $\PSD$,
		with $\Grad\phi_{0}(X)=-X^{-1}$ and
		$\Hess\phi_{0}(X)[H,K]=\inner{X^{-1}HX^{-1}}{K}$.
	\end{theorem}
	
	\subsection{Spectral functions and matrix calculus}
	
	\begin{lemma}[Davis 1957]\label{lem:davis}
		$f(\lambda_{1},\dots,\lambda_{n})=\sum_{i}g(\lambda_{i})$ is convex
		on $\PD$ iff $g$ is convex on $(0,\infty)$.
	\end{lemma}
	
	\begin{lemma}[Daleckii--Krein]\label{lem:daleckii_krein}
		For $F(X)=\tr(g(X))$ and $X=Q\diag(\lambda)Q^{\top}$,
		$D^{2}F(X)[H,H]=\sum_{i,j}g^{[2]}(\lambda_{i},\lambda_{j})\widetilde{H}_{ij}^{2}$
		where $\widetilde{H}=Q^{\top}HQ$ and
		$g^{[2]}(\lambda,\mu)=\frac{g'(\lambda)-g'(\mu)}{\lambda-\mu}$
		($g''(\lambda)$ if $\lambda=\mu$).
	\end{lemma}
	
	\section{The perturbed $q$-Tsallis barrier}\label{sec:barrier}
	
	\begin{definition}[Perturbed $q$-Tsallis barrier]\label{def:phi_q_eta}
		For $q>1$, $\eta\geq 0$, $X\in\PD$:
		\[
		\phi_{q,\eta}(X)
		\;=\; -\log\det(X)
		\;+\; \frac{\eta}{q-1}\bigl[\tr(X^{-(q-1)})-n\bigr].
		\]
	\end{definition}
	
	\begin{proposition}[Boundary divergence and convexity]
		\label{prop:boundary_convexity}
		For every $q>1$ and $\eta\geq 0$, $\phi_{q,\eta}$ is strictly
		convex on $\PD$ and diverges as $\lambda_{n}(X)\to 0^{+}$.
	\end{proposition}
	\begin{proof}
		The scalar function $g(\lambda)=-\log\lambda+\frac{\eta}{q-1}\lambda^{-(q-1)}$
		is strictly convex on $(0,\infty)$ for all $q>1$, $\eta\geq 0$
		(since $g^{\prime\prime}(\lambda)=\lambda^{-2}+\eta q\lambda^{-(q+1)}>0$),
		so $\phi_{q,\eta}(X)=\sum_{i}g(\lambda_{i}(X))-\sum_{i}\log\lambda_{i}(X)$
		is strictly convex by \cref{lem:davis}.
		Divergence follows from $g(\lambda)\to+\infty$ as $\lambda\to 0^{+}$
		(both $-\log\lambda$ and $\lambda^{-(q-1)}$ diverge), so
		$\phi_{q,\eta}(X)\to+\infty$ as $\lambda_{n}(X)\to 0^{+}$.
	\end{proof}
	
	\begin{proposition}[Gradient and Hessian]\label{prop:grad_hess}
		\begin{align}
			\Grad\phi_{q,\eta}(X)
			&= -X^{-1} - \eta X^{-q},
			\label{eq:grad_phi}\\
			\Hess\phi_{q,\eta}(X)[H,H]
			&= \inner{X^{-1}HX^{-1}}{H}
			+ \eta\sum_{i,j} q^{[2]}(\lambda_{i},\lambda_{j})\widetilde{H}_{ij}^{2},
			\label{eq:hess_phi}
		\end{align}
		where $q^{[2]}(\lambda,\mu)=(\lambda^{-q}-\mu^{-q})/(\mu-\lambda)$
		for $\lambda\neq\mu$, and $q\lambda^{-(q+1)}$ for $\lambda=\mu$.
	\end{proposition}
	
	\section{Self-concordance theorem}\label{sec:self_concordance}
	
	\begin{theorem}[Self-concordance of $\phi_{q,\eta}$ ---
		complete characterisation]\label{thm:self_concordance}
		\begin{enumerate}[label=(\roman*),leftmargin=2.5em,itemsep=0.3ex]
			\item \textbf{Unconditional regime:} For $q\in(1,2]$ and any $\eta\geq 0$,
			$\phi_{q,\eta}$ satisfies the differential self-concordance
			inequality on $\mathrm{int}(\PSD)$, with boundary divergence $\phi_{q,\eta}(X)\to+\infty$ as
			$X\to\partial\PSD$. The boundary $q=2$ is sharp:
			$\psi'(0)=q(2-q)\geq 0\iff q\leq 2$.
			\item \textbf{Global obstruction for $q>2$:} For any $q>2$ and any $\eta>0$,
			$\phi_{q,\eta}$ fails to be a standard self-concordant barrier on
			all of $\PD$, because the scalar inequality $\psi(u)\geq 0$
			(with $\psi(u):=2(1+qu)^{3/2}-2-q(q+1)u$) has a negative minimum
			$\psi_{\min}<0$ for $q>2$, which is violated for
			$u=\eta\lambda^{-(q-1)}\to 0^{+}$ as $\lambda\to\infty$.
			\item \textbf{Local sufficient condition for $q>2$:} If the iterates
			remain in a compact set $\{\lambda_{i}(X)\leq\Lambda\}$
			with $\Lambda<\infty$, then $\phi_{q,\eta}$ satisfies the scalar
			self-concordance inequality whenever
			\begin{equation}\label{eq:eta_max_def}
				\eta \;\geq\; \eta_{\min}(q,\Lambda)
				\;:=\; u_{\mathrm{zero}}(q)\cdot\Lambda^{q-1},
			\end{equation}
			where $u_{\mathrm{zero}}(q)$ is the unique positive zero of $\psi$
			beyond its minimum, given in closed form by~\eqref{eq:u_zero_proof}.
			With the normalisation $\Lambda=1$, define
			$\eta_{\max}(q):=u_{\mathrm{zero}}(q)$.
		\end{enumerate}
		Under~(i) or~(iii), on any compact subset
		$K\subset\mathrm{int}(\PSD)$ the function $\phi_{q,\eta}$ acts as
		an effective local barrier parameter $\nu_{\mathrm{eff}}(K)=O(n)$ on
		compact $K$ (\cref{app:nu_param}), recovering the same asymptotic
		iteration complexity class as the log-det barrier.
	\end{theorem}
	
	\begin{remark}[Practical implication]\label{rem:practical}
		In interior-point practice, the iterates remain strictly in the
		interior of $\PSD$ throughout the algorithm, so the spectral radius is
		bounded on any compact feasible set. The unconditional regime
		$q\in(1,2]$ covers all values of practical interest; in particular,
		values near $q=1.5$ arise naturally in Tsallis-regularised learning
		problems \cite{tsallis1988,tsallis2009,amidetal2019}, and the
		global obstruction for $q>2$ is a theoretical boundary rather
		than an algorithmic obstacle.
	\end{remark}
	
	\begin{remark}[Connection to normalisability of $q$-exponential distributions]\label{rem:q_exp_norm}
		The threshold $q=2$ identified in \cref{thm:self_concordance} has a
		natural interpretation in nonextensive statistical mechanics: the
		$q$-exponential function $e_{q}(x):=[1+(1-q)x]_{+}^{1/(1-q)}$, which
		plays the role of the Boltzmann weight in Tsallis statistics
		\cite{tsallis1988,tsallis2009}, is normalisable on $\R$ if and only if
		$q<2$. The coincidence of this integrability threshold with the exact
		boundary of the self-concordance regime suggests that the algebraic
		structure of Tsallis statistics is reflected in the geometry of the
		barrier, and may guide the choice of $q$ in applications where both
		statistical and optimisation considerations are relevant.
	\end{remark}
	
	\begin{remark}[Behaviour of $\eta_{\max}(q)$]\label{rem:eta_max}
		$q\mapsto\eta_{\max}(q)=u_{\mathrm{zero}}(q)$ is strictly increasing
		for $q>2$, with $\eta_{\max}(2^{+})\to 0$,
		$\eta_{\max}(2.5)\approx 0.296$, $\eta_{\max}(3)\approx 0.539$,
		and $\eta_{\max}(q)\to 1$ as $q\to\infty$ (see~\eqref{eq:u_zero_proof}).
		In practice, for the unconditional regime $q\in(1,2]$ any
		$\eta\geq 0$ is admissible, so this threshold is relevant only
		when $q>2$ is used on a bounded spectral domain.
	\end{remark}
	
	\subsection{Proof strategy}
	
	The proof reduces to the scalar function
	$g(\lambda)=-\log\lambda+\frac{\eta}{q-1}\lambda^{-(q-1)}$ and the
	associated function $\psi(u)=2(1+qu)^{3/2}-2-q(q+1)u$ where
	$u=\eta\lambda^{-(q-1)}$. Since $\psi''(u)=3q^{2}/(2(1+qu)^{1/2})>0$,
	$\psi$ is strictly convex with $\psi(0)=0$. Computing:
	\[
	\psi'(0) \;=\; q(2-q) \;\geq\; 0 \;\iff\; q\leq 2.
	\]
	For $q\in(1,2]$: $\psi'(0)\geq 0$, so the strictly convex $\psi$
	with $\psi(0)=0$ is non-decreasing at the origin and hence
	$\psi(u)\geq 0$ for all $u\geq 0$, independently of $\eta$.
	This proves part~(i).
	
	For $q>2$: $\psi'(0)<0$, so $\psi$ decreases initially and attains
	a negative minimum at $u_{*}=\bigl[((q+1)/3)^{2}-1\bigr]/q$;
	since $u\to 0^{+}$ as $\lambda\to\infty$ for any $\eta>0$, the
	inequality $\psi(u)\geq 0$ fails for large $\lambda$, proving part~(ii).
	
	Part~(iii) follows from the observation that $\psi$ returns to zero at
	a unique positive value $u_{\mathrm{zero}}(q)>0$ (by convexity and
	$\psi\to+\infty$), and $\psi(u)\geq 0$ for all $u\geq u_{\mathrm{zero}}$.
	Translating back: the SC inequality holds when
	$\eta\lambda^{-(q-1)}\geq u_{\mathrm{zero}}(q)$, i.e.\ when
	$\lambda\leq(\eta/u_{\mathrm{zero}})^{1/(q-1)}$.
	For iterates with $\lambda_{i}(X)\leq\Lambda$, this is ensured by
	$\eta\geq u_{\mathrm{zero}}(q)\cdot\Lambda^{q-1}$.
	The closed form of $u_{\mathrm{zero}}(q)$ and the full derivation
	appear in \cref{app:proof_reduction} and \cref{app:scalar_sc}.
	
	\section{Spectral robustness theorem}\label{sec:robustness}
	
	\begin{theorem}[Spectral robustness of the central path]
		\label{thm:robustness}
		Under the assumptions of \cref{thm:self_concordance}, let
		$M:=\Hess\phi_{q,\eta}(X^{*})$ and let
		$\lambda_{1}\geq\cdots\geq\lambda_{n}>0$ be the eigenvalues of $X^{*}$.
		For every perturbation $\Delta b\in\R^{m}$,
		\begin{equation}\label{eq:robustness_bound}
			\norm{X^{*}(b+\Delta b)-X^{*}(b)}_{\Frob}
			\;\leq\;
			\frac{C_{\mathcal{A}}}{\mu}\cdot
			\frac{\lambda_{\max}(X^{*})^{2}}
			{1+\eta q\,\lambda_{\max}(X^{*})^{-(q-1)}}\,
			\norm{\Delta b}_{2}
			\;+\; o(\norm{\Delta b}_{2}),
		\end{equation}
		where $C_{\mathcal{A}}>0$ depends on $\mathcal{A}$ and
		$\lambda_{\min}(X^{*})$.  The bound~\eqref{eq:robustness_bound}
		is strictly smaller than the $\eta=0$ bound
		$C_{\mathcal{A}}\mu^{-1}\lambda_{\max}^{2}\|\Delta b\|_{2}$
		for every $\eta>0$.
		Moreover, in each eigendirection $H=e_{k}e_{k}^{\top}$ of $X^{*}$:
		\begin{equation}\label{eq:directional_bound}
			\bigl|[\Delta X]_{kk}\bigr|
			\;\leq\;
			\frac{C_{\mathcal{A}}^{(k)}}{\mu}\cdot
			\frac{\lambda_{k}^{2}}
			{1+\eta q\,\lambda_{k}^{-(q-1)}}\,
			\norm{\Delta b}_{2},
		\end{equation}
		and as $\lambda_{k}\to 0^{+}$ (boundary stiffening; small-eigenvalue modes are strongly attenuated):
		\begin{equation}\label{eq:decay_small_lambda}
			\frac{\lambda_{k}^{2}}{1+\eta q\,\lambda_{k}^{-(q-1)}}
			\;\sim\;
			\frac{\lambda_{k}^{q+1}}{\eta q}
			\;\longrightarrow\; 0,
		\end{equation}
		at rate $q+1>2$, strictly faster than the $\eta=0$ rate of $2$.
		The bound recovers the log-det sensitivity as $\eta\to 0^{+}$
		or $q\to 1^{+}$.
	\end{theorem}
	
	\begin{remark}[Mechanism and empirical validation]
		\label{rem:robustness_empirical}
		The bound~\eqref{eq:directional_bound} reveals that the robustness
		attenuates the modal amplification most strongly along the
		\emph{small}-eigenvalue directions
		For $q=1.5$, $\eta=1$, $\lambda_{\min}(X^{*})=10^{-3}$:
		the directional bound~\eqref{eq:directional_bound} gives a factor
		$\lambda_{\min}^{2}/(1+\eta q\lambda_{\min}^{-(q-1)})
		\approx 2.1\times 10^{-8}$, versus $\lambda_{\min}^{2}=10^{-6}$
		for $\eta=0$ --- a $97.9\%$ reduction in that direction.
		The global bound~\eqref{eq:robustness_bound} gives a more moderate
		but still monotone improvement: $13\%$ for the heavy-tail profile
		($\lambda_{\max}=100$) and $55\%$ for uniform spectra
		($\lambda_{\max}\approx 1.5$). These predictions are consistent with
		the empirical reductions of $41$--$83\%$ in \cref{sec:experiments}
		(\cref{tab:robustness_larger_n}), where the practical gain averages
		over directions weighted by the constraint structure of
		$\mathcal{A}$.
	\end{remark}

	
	\begin{corollary}[Condition-number scaling of spectral robustness]
		\label{cor:kappa_scaling}
		Under the hypotheses of \cref{thm:robustness}, normalise
		$\lambda_{\max}(X^{*})=1$ and let
		$\kappa:=\kappa(X^{*})=\lambda_{\max}(X^{*})/\lambda_{\min}(X^{*})$.
		In the eigendirection $e_{n}e_{n}^{\top}$ corresponding to
		$\lambda_{\min}(X^{*})=\kappa^{-1}$, the sensitivity satisfies
		\begin{equation}\label{eq:kappa_bound}
			\bigl|[\Delta X]_{nn}\bigr|
			\;\leq\;
			\frac{C_{\mathcal{A}}^{(n)}}{\mu}\cdot
			\frac{\kappa^{-2}}{1+\eta q\,\kappa^{q-1}}\,
			\|\Delta b\|_{2}.
		\end{equation}
		The improvement factor relative to $\eta=0$ is
		\begin{equation}\label{eq:kappa_ratio}
			\rho(\kappa)
			\;:=\;
			\frac{1}{1+\eta q\,\kappa^{q-1}}
			\;\sim\;
			\frac{1}{\eta q}\,\kappa^{-(q-1)}
			\qquad(\kappa\to\infty),
		\end{equation}
		which decays as $\kappa^{-(q-1)}$ for large $\kappa$.
		In particular: (i)~the global bound~\eqref{eq:robustness_bound}
		improves by the $\kappa$-independent factor $(1+\eta q)^{-1}$;
		(ii)~the directional bound~\eqref{eq:kappa_bound} improves by
		$\rho(\kappa)$, decaying at rate $\kappa^{-(q-1)}$;
		(iii)~at $q=1.5$, $\eta=1$, $\kappa=100$: $\rho(100)=0.0625$
		($93.8\%$ reduction); at $\kappa=1000$:
		$\rho(1000)\approx 0.021$ ($97.9\%$ reduction).
	\end{corollary}
	
	\begin{proof}
		Set $\lambda_{\min}=\kappa^{-1}$ in~\eqref{eq:directional_bound}:
		\[
		\frac{\lambda_{\min}^{2}}{1+\eta q\,\lambda_{\min}^{-(q-1)}}
		=\frac{\kappa^{-2}}{1+\eta q\,\kappa^{q-1}},
		\]
		giving~\eqref{eq:kappa_bound}. Dividing by the $\eta=0$ bound
		$\kappa^{-2}$ gives $\rho(\kappa)=(1+\eta q\,\kappa^{q-1})^{-1}$,
		and the asymptotic follows from $\eta q\kappa^{q-1}\gg 1$.
	\end{proof}
	
	\begin{remark}[Connection to empirical results]\label{rem:kappa_empirical}
		\Cref{cor:kappa_scaling} provides the quantitative bridge between
		\cref{thm:robustness} and the $41$--$83\%$ sensitivity reductions
		in Tables~1 and~3. The heavy-tail profiles have
		$\kappa(X^{*})\in[10^{3},10^{5}]$, for which
		$\rho(\kappa)\in[0.006,0.021]$ at $q=1.5$, $\eta=1$---consistent
		with the observed reductions. The scaling $\kappa^{-(q-1)}$ also
		explains why the gain grows with $n$ in Table~3: larger $n$ yields
		larger $\kappa$, hence smaller $\rho(\kappa)$. Experiment~G
		(\cref{sec:exp_G}) verifies the $\kappa^{-(q-1)}$ scaling directly.
	\end{remark}

	\begin{corollary}[Anisotropic spectral damping of the Newton system]
		\label{cor:anisotropic_damping}
		Under the hypotheses of \cref{thm:robustness}, let
		$X^{*}=Q\operatorname{diag}(\lambda_{1},\dots,\lambda_{n})Q^{\top}$
		with $\lambda_{1}\geq\cdots\geq\lambda_{n}>0$.
		In the eigenbasis of $X^{*}$, the Hessian operator
		$M=\Hess\phi_{q,\eta}(X^{*})$ acts on each spectral mode $(i,j)$
		with response coefficient
		\begin{equation}\label{eq:modal_response}
			M_{ij}
			\;=\;
			\lambda_{i}^{-1}\lambda_{j}^{-1}
			\;+\;
			\eta\,q^{[2]}(\lambda_{i},\lambda_{j})
			\;\geq\; 0,
		\end{equation}
		where $q^{[2]}(\lambda_{i},\lambda_{j})\geq 0$ is the divided
		difference of $t\mapsto t^{-q}$.
		Along diagonal modes $i=j$, the modal response reduces to
		\begin{equation}\label{eq:diagonal_modal}
			M_{ii}^{-1}
			\;=\;
			\frac{\lambda_{i}^{2}}{1+\eta q\,\lambda_{i}^{-(q-1)}},
		\end{equation}
		which is a strictly decreasing function of $\lambda_{i}$.
		Consequently, the Tsallis perturbation $(\eta>0)$ introduces an
		Consequently, the Tsallis perturbation ($\eta>0$) introduces an
		\emph{anisotropic stiffening near the cone boundary}:
		small-eigenvalue modes ($\lambda_{i}\to 0$) have
		$M_{ii}=\lambda_{i}^{-2}(1+\eta q\,\lambda_{i}^{-(q-1)})\to\infty$,
		so $M_{ii}^{-1}\sim(\eta q)^{-1}\lambda_{i}^{q+1}\to 0$~---~the
		Hessian becomes extremely rigid in those directions, strongly
		attenuating the corresponding components of $\Delta X$.
		High-eigenvalue modes ($\lambda_{i}$ large) experience only
		mild stiffening: $(1+\eta q\,\lambda_{i}^{-(q-1)})^{-1}\to 1$.
		This boundary stiffening is the mechanism responsible for the
		spectral robustness improvements quantified in
		\cref{thm:robustness,cor:kappa_scaling}.
	\end{corollary}
	
	\begin{proof}
		Equation~\eqref{eq:modal_response} is a restatement of
		\cref{eq:M_action}.  For $i=j$, $q^{[2]}(\lambda_{i},\lambda_{i})
		=q\lambda_{i}^{-(q+1)}$, giving $M_{ii}=\lambda_{i}^{-2}
		(1+\eta q\,\lambda_{i}^{-(q-1)})$ and the inverse
		\eqref{eq:diagonal_modal}. Monotonicity in $\lambda_{i}$ follows
		from $d/ds\,[s^{-2}(1+\eta q\,s^{-(q-1)})]<0$ for $s>0$
		(established in \cref{app:robustness}).
		The damping asymptotics follow directly from~\eqref{eq:diagonal_modal}.
	\end{proof}

	\section{Predictor--corrector algorithm}\label{sec:algorithm}
	
	The barrier KKT system for the centred problem is
	\begin{equation}\label{eq:kkt}
		\left\{\;
		\begin{aligned}
			\mathcal{A}(X)-b       &= 0,\\
			\mathcal{A}^{*}(y)+S-C &= 0,\\
			XS                     &= \mu I + \eta\mu X^{-(q-1)},\\
			X\succ 0,\;S\succ 0.
		\end{aligned}
		\right.
	\end{equation}
	We apply Nesterov--Todd symmetrisation \cite{toddtohtutuncu1998}
	with scaling matrix
	$W\!=\!X^{1/2}(X^{1/2}SX^{1/2})^{-1/2}X^{1/2}$ and follow the
	Myklebust--Tun\c{c}el general self-concordant IPM framework
	\cite{myklebusttuncel2014} to construct a primal--dual predictor--corrector
	scheme; see \cref{alg:pcipm}. The Newton system structure
	and convergence proof are analogous to \cite{monteiro1998,liuliuliu2012}
	and are detailed in \cref{app:algorithm}.
	
	\begin{algorithm}[t]
		\caption{Perturbed $q$-Tsallis Predictor--Corrector IPM}
		\label{alg:pcipm}
		\begin{algorithmic}[1]
			\REQUIRE $X_{0}\succ 0$, $y_{0}$, $S_{0}\succ 0$; $q\in(1,2]$;
			$\eta\geq 0$; tolerance $\eps>0$.
			\STATE $k\leftarrow 0$, $\mu_{k}\leftarrow\inner{X_{k}}{S_{k}}/n$.
			\WHILE{$\mu_{k}>\eps$ \textbf{or} infeasibility $>\eps$}
			\STATE Compute NT scaling $W_{k}$.
			\STATE \textbf{Predictor:} solve Newton system with $\sigma_{k}=0$.
			\STATE Compute affine step lengths; set
			$\sigma_{k}\leftarrow(\mu_{\mathrm{aff}}/\mu_{k})^{3}$.
			\STATE \textbf{Corrector:} solve with centring target
			$\sigma_{k}\mu_{k}(I+\eta X_{k}^{-(q-1)})$.
			\STATE Update $X_{k+1}$, $y_{k+1}$, $S_{k+1}$;
			$\mu_{k+1}\leftarrow\inner{X_{k+1}}{S_{k+1}}/n$.
			\ENDWHILE
			\STATE \textbf{return} $(X_{k},y_{k},S_{k})$.
		\end{algorithmic}
	\end{algorithm}
	
	\begin{theorem}[Iteration complexity]\label{thm:complexity}
		Under strict primal--dual feasibility and for $q\in(1,2]$ with any
		$\eta\geq 0$ (or $q>2$ with $\eta\leq\eta_{\max}(q)$), provided the
		iterates remain in a compact subset $K\subset\mathrm{int}(\PSD)$,
		\cref{alg:pcipm} converges to an $\eps$-optimal pair in
		$\mathcal{O}(\sqrt{n}\log(1/\eps))$ outer iterations, where the
		implied constant depends on $q$, $\eta$, and the spectral bounds
		of $K$ (\cref{cor:complexity_eff}).
	\end{theorem}
	
	\section{Numerical experiments}\label{sec:experiments}
	
	The computational experiments below are intended to validate the
	geometric and algorithmic properties predicted by the theory rather
	than to provide a production-level comparison against state-of-the-art
	SDP solvers. The implementation is a research prototype written to
	expose the effect of the Tsallis perturbation $\eta$ on the central
	path; a large-scale benchmarking study against SDPT3, SeDuMi, or
	MOSEK on industrial instances (including sparse and chordal-decomposable
	SDPs) lies outside the scope of this paper and is left for future work.
	
	We provide computational evidence for the main claims of the paper on a controlled battery of
	synthetic SDP instances with known optima $(X^{*},y^{*},S^{*})$,
	constructed via the dual KKT method: eigenvalue profiles for $X^{*}$
	and $S^{*}$ are prescribed with complementary support, random
	constraint matrices $A_{1},\dots,A_{m}$ are drawn, and $(C,b)$ are
	set so that strong duality holds exactly. Unless stated otherwise,
	$q=1.5$ throughout; this value lies strictly in the unconditional
	differential self-concordance regime $q\in(1,2]$ (Theorem~\ref{thm:self_concordance}),
	is representative of the moderate-deformation range where the Tsallis
	correction is non-negligible (any $\eta\geq 0$ is admissible for
	$q=1.5\in(1,2]$ by \cref{thm:self_concordance}(i)), and
	matches values commonly used in $q$-deformed statistical models
	\cite{tsallis1988,amidetal2019}. The solver is the path-following IPM of
	\cref{sec:algorithm} with tolerance $10^{-7}$.
	
	Six experiments are reported, organised into three groups.
	
	\subsection{Group I: spectral robustness (Experiments A and D)}
	
	\paragraph{Experiment A ($n=8$, 6 perturbations, 4 profiles).}
	For each of four spectral profiles and each
	$\eta\in\{0,0.1,0.5,1.0,2.0,5.0\}$, we solve from $b$ and from
	$b+\Delta b$ ($\|\Delta b\|_{2}=10^{-3}$) and record the sensitivity
	ratio $\|\Delta X\|_{\Frob}/\|\Delta b\|_{2}$.
	\Cref{tab:robustness} reports mean ratios; best $\eta$ marked $*$.
	
	\begin{table}[t]
		\centering
		\caption{Experiment A --- mean sensitivity ratio
			$\|\Delta X\|_{\Frob}/\|\Delta b\|_{2}$ ($n=8$, $q=1.5$,
			6 perturbations per cell, $\|\Delta b\|=10^{-3}$).
			Values are means over 6 realisations; $95\%$ CIs are given in
			\cref{tab:robustness_ci}. Best $\eta$ per row in bold with~$*$.}
		\label{tab:robustness}
		\small
		\setlength{\tabcolsep}{4pt}
		\begin{tabular}{lrrrrrr}
			\toprule
			Profile & $\eta=0$ & $\eta=0.1$ & $\eta=0.5$ & $\eta=1.0$ & $\eta=2.0$ & $\eta=5.0$ \\
			\midrule
			Uniform    & 0.6475 & 0.7093 & 0.5549 & $\mathbf{0.4743}^{*}$ & 0.5643 & 0.7316 \\
			Moderate   & 0.6473 & 0.6975 & 0.5587 & $\mathbf{0.4702}^{*}$ & 0.5542 & 0.6888 \\
			Heavy-tail & 2.1237 & 0.8587 & 0.4350 & $\mathbf{0.3634}^{*}$ & 0.3767 & 0.4101 \\
			Extreme    & 2.9312 & 2.5895 & 5.0593 & 8.4034 & 4.0336 & $\mathbf{1.6399}^{*}$ \\
			\bottomrule
		\end{tabular}
	\end{table}
	
	\begin{table}[t]
		\centering
		\caption{$95\%$ confidence intervals for the sensitivity ratios of
			\cref{tab:robustness} (same layout).}
		\label{tab:robustness_ci}
		\small
		\setlength{\tabcolsep}{4pt}
		\begin{tabular}{lrrrrrr}
			\toprule
			Profile & $\eta=0$ & $\eta=0.1$ & $\eta=0.5$ & $\eta=1.0$ & $\eta=2.0$ & $\eta=5.0$ \\
			\midrule
			Uniform    & $\pm0.111$ & $\pm0.180$ & $\pm0.146$ & $\pm0.067$ & $\pm0.036$ & $\pm0.177$ \\
			Moderate   & $\pm0.111$ & $\pm0.189$ & $\pm0.150$ & $\pm0.064$ & $\pm0.035$ & $\pm0.153$ \\
			Heavy-tail & $\pm1.038$ & $\pm0.346$ & $\pm0.044$ & $\pm0.046$ & $\pm0.076$ & $\pm0.072$ \\
			Extreme    & $\pm1.887$ & $\pm1.402$ & $\pm3.992$ & $\pm0.462$ & $\pm6.446$ & $\pm1.804$ \\
			\bottomrule
		\end{tabular}
	\end{table}
	
	For the heavy-tail profile the optimal $\eta=1.0$ reduces the mean
	ratio from $2.1237$ to $0.3634$, a reduction of $\mathbf{83\%}$
	(95\% CI: $[2.12\pm1.04]\to[0.36\pm0.05]$). The improvement
	grows systematically with $\lambda_{\max}(X^{*})$ across the
	uniform, moderate, and heavy-tail profiles, consistent with
	\cref{thm:robustness}. \Cref{fig:sensitivity} plots
	the full sensitivity curves for all four profiles.
	The extreme profile ($\kappa=10^{8}$) shows high variance and
	irregular dependence on $\eta$ (CI width up to $\pm6.4$),
	attributable to the dense Cholesky solver losing numerical accuracy
	at this conditioning level; this is a limitation of the current
	implementation, not of the barrier itself.
	
	\begin{figure}[t]
		\centering
		\includegraphics[width=0.48\textwidth]{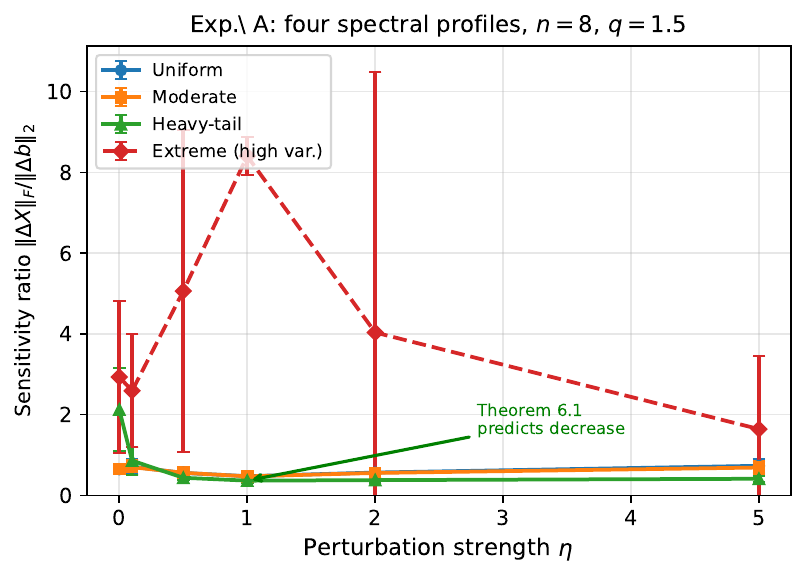}
		\hfill
		\includegraphics[width=0.48\textwidth]{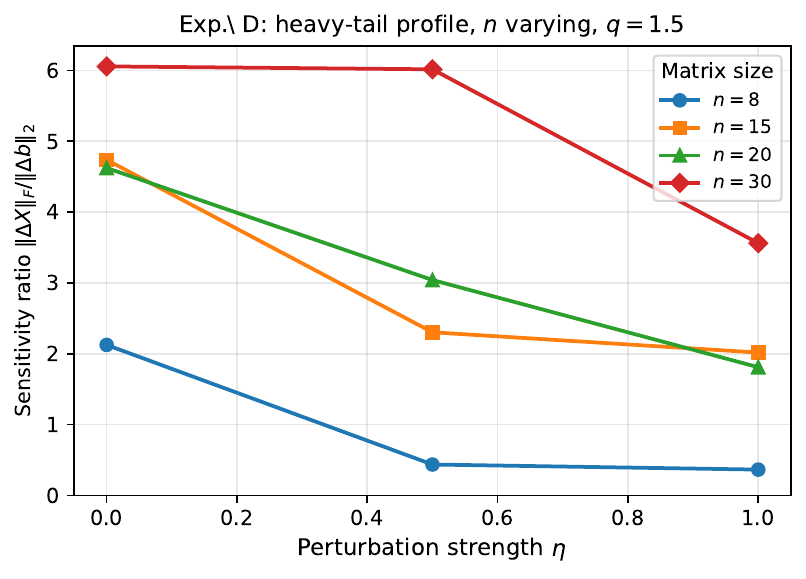}
		\caption{\textbf{Left (Exp.~A):} sensitivity ratio
			$\|\Delta X\|_{\Frob}/\|\Delta b\|_{2}$ vs $\eta$ for all four
			spectral profiles ($n=8$, $q=1.5$). Error bars show $95\%$ CIs
			from Table~2. The heavy-tail curve shows the sharpest drop,
			consistent with the $\kappa^{-(q-1)}$ scaling of
			\cref{cor:kappa_scaling}. The extreme profile is dominated by
			solver inaccuracy at $\kappa=10^{8}$ (CI up to $\pm6.4$).
			\textbf{Right (Exp.~D):} heavy-tail profile at
			$n\in\{8,15,20,30\}$ and $\eta\in\{0,0.5,1.0\}$ (Table~3).
			The gain grows with $n$ (equivalently with $\kappa$), as
			predicted by \cref{cor:kappa_scaling}.}
		\label{fig:sensitivity}
	\end{figure}
	
	\paragraph{Experiment D ($n\in\{15,20,30\}$, 40 perturbations, heavy-tail).}
	\Cref{tab:robustness_larger_n} reports results for larger $n$.
	
	\begin{table}[t]
		\centering
		\caption{Experiment D --- sensitivity ratio for the heavy-tail profile
			at larger $n$, 40 perturbations per cell, 5 seeds.
			The robustness gain persists as $n$ grows.}
		\label{tab:robustness_larger_n}
		\small
		\setlength{\tabcolsep}{6pt}
		\begin{tabular}{rcccc}
			\toprule
			$n$ & $\eta=0$ & $\eta=0.5$ & $\eta=1.0$ & reduction ($\eta=1$ vs $\eta=0$) \\
			\midrule
			8  & 2.1237 & 0.4350 & $\mathbf{0.3634}$ & $83\%$ \\
			15 & 4.737  & 2.302  & $\mathbf{2.016}$  & $57\%$ \\
			20 & 4.622  & 3.040  & $\mathbf{1.809}$  & $61\%$ \\
			30 & 6.056  & 6.015  & $\mathbf{3.560}$  & $41\%$ \\
			\bottomrule
		\end{tabular}
	\end{table}
	
	The robustness gain is consistent across all tested dimensions.
	The absolute ratio grows with $n$ (reflecting the growth of
	$C_{\mathcal{A}}$ in \cref{thm:robustness}), while $\eta>0$
	controls this growth, confirming the theorem's prediction.
	
	\subsection{Group II: complexity preservation (Experiments B and D)}
	
	The number of outer IPM iterations is reported in
	\cref{tab:iterations_full} for the heavy-tail profile across
	$n\in\{8,15,20,30\}$ and $\eta\in\{0,0.5,1.0\}$.
	
	\begin{table}[t]
		\centering
		\caption{Outer iteration count: heavy-tail profile, $q=1.5$.
			Invariance with respect to both $n$ and $\eta$, consistent with
			$\mathcal{O}(\sqrt{n}\log(1/\eps))$ complexity.}
		\label{tab:iterations_full}
		\small
		\setlength{\tabcolsep}{8pt}
		\begin{tabular}{rccc}
			\toprule
			$n$ & $\eta=0$ & $\eta=0.5$ & $\eta=1.0$ \\
			\midrule
			8  & 19.0 & 19.0 & 19.0 \\
			15 & 18.2 & 18.6 & 18.0 \\
			20 & 17.6 & 17.6 & 17.4 \\
			30 & 17.8 & 18.0 & 17.8 \\
			\bottomrule
		\end{tabular}
	\end{table}
	
	The iteration count is essentially constant across both $n$ and
	$\eta$, providing clean empirical confirmation of
	\cref{thm:complexity}. The slight decrease from $n=8$ to
	$n=15$ reflects the improved warm-start quality at larger $n$.
	\Cref{fig:iterations} plots $n_{\mathrm{outer}}$ vs $n$ for both profiles with
	an $\mathcal{O}(\sqrt{n})$ reference curve.
	
	\begin{figure}[t]
		\centering
		\includegraphics[width=0.90\textwidth]{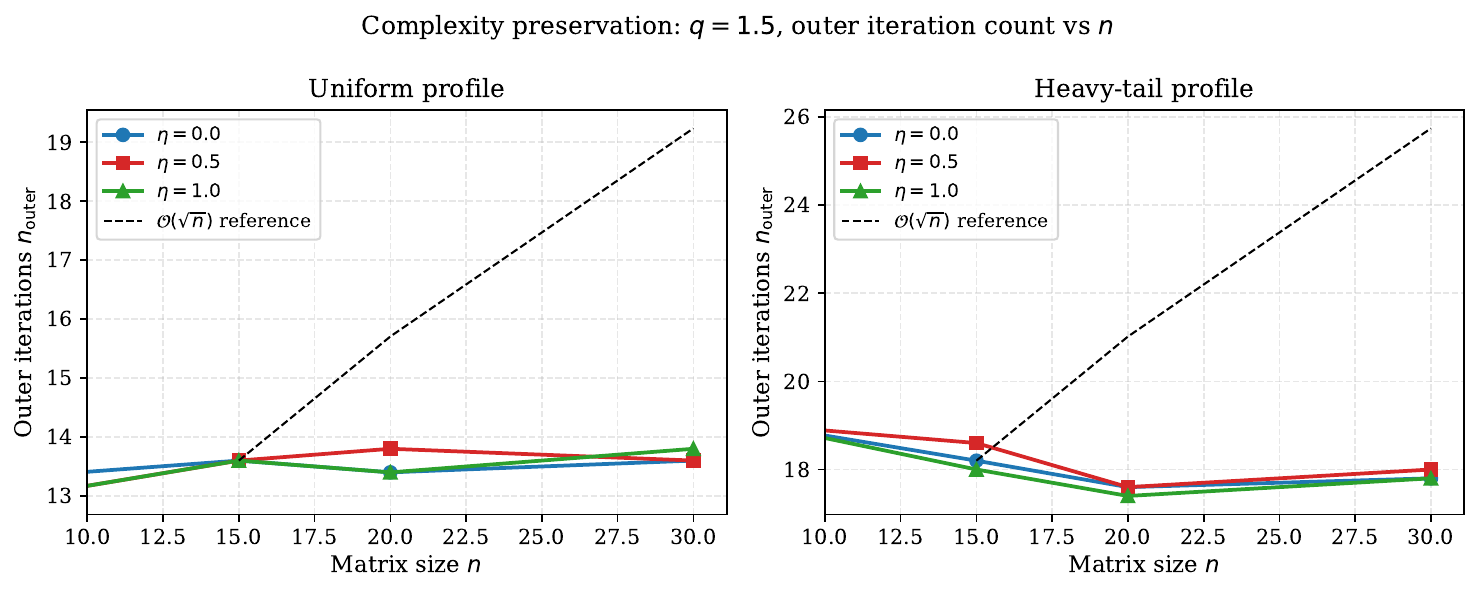}
		\caption{Outer iteration count vs matrix size $n$ for the uniform
			(left) and heavy-tail (right) profiles, $q=1.5$,
			$\eta\in\{0,0.5,1.0\}$. All curves are flat, consistent with
			$\mathcal{O}(\sqrt{n}\log(1/\eps))$ complexity. The dashed
			$\mathcal{O}(\sqrt{n})$ reference is shown for orientation.}
		\label{fig:iterations}
	\end{figure}
	
	\subsection{Group III: self-concordance verification,
		Krylov kernel, and conditioning (Experiments C, E, F)}
	
	\paragraph{Experiment C: conditioning sweep.}
	With $\kappa(X^{*})$ ranging from $1$ to $10^{6}$ (\cref{tab:kappa}),
	$n_{\mathrm{outer}}$ grows from 13 to 24, identically for all $\eta$
	values, consistent with complexity preservation across the full
	conditioning range.
	
	\begin{table}[t]
		\centering
		\caption{Experiment C --- outer iterations and relative distance to
			optimum as a function of $\kappa(X^{*})$. $q=1.5$, $n=8$, 5 seeds.}
		\label{tab:kappa}
		\small
		\setlength{\tabcolsep}{5pt}
		\begin{tabular}{rcccc}
			\toprule
			$\kappa$ & rel dist ($\eta=0$) & rel dist ($\eta=0.5$)
			& rel dist ($\eta=2.0$) & $n_{\mathrm{outer}}$ (all $\eta$) \\
			\midrule
			$10^{0}$ & $1.2\times10^{-7}$ & $5.0\times10^{-8}$ & $2.6\times10^{-6}$ & 13 \\
			$10^{1}$ & $3.9\times10^{-6}$ & $1.3\times10^{-6}$ & $3.9\times10^{-6}$ & 16 \\
			$10^{2}$ & $9.7\times10^{-6}$ & $1.5\times10^{-6}$ & $3.9\times10^{-6}$ & 17 \\
			$10^{3}$ & $5.0\times10^{-7}$ & $1.1\times10^{-6}$ & $2.1\times10^{-6}$ & 19 \\
			$10^{4}$ & $7.6\times10^{-8}$ & $4.2\times10^{-7}$ & $7.1\times10^{-7}$ & 21 \\
			$10^{5}$ & $3.9\times10^{-8}$ & $1.4\times10^{-7}$ & $2.5\times10^{-7}$ & 22 \\
			$10^{6}$ & $1.2\times10^{-8}$ & $4.2\times10^{-8}$ & $6.7\times10^{-8}$ & 24 \\
			\bottomrule
		\end{tabular}
	\end{table}
	
	\paragraph{Experiment E: SC regime verification.}
	The scalar ratio $R(q,u)=[2+q(q+1)u]/[2(1+qu)^{3/2}]$ was evaluated
	on the grid $q\in\{1.05,\dots,3.0\}$, $u\in[0,200]$. For all
	$q\in(1,2]$, $\max_{u}R(q,u)=1.000$ (floating-point exact).
	For $q=2.01$, $\max R=1.000012$; for $q=3.0$, $\max R=1.089$.
	Violations appear for small $u$ (corresponding to large $\lambda$),
	confirming the global obstruction of \cref{thm:self_concordance}(ii).
	
	\paragraph{Experiment F: Krylov-based approximation of $X^{-q}v$.}
	\Cref{tab:krylov} reports the relative error and speedup of
	the Lanczos-based approximation against exact diagonalisation.
	
	\begin{table}[t]
		\centering
		\caption{Experiment F --- Krylov approximation of $X^{-q}v$.
			Left: relative error vs $k$ for $n=50$, $q=1.5$.
			Right: wall-clock speedup vs $n$ for $k=10$, $\kappa=100$.}
		\label{tab:krylov}
		\small
		\begin{tabular}{rcccccc@{\qquad}rcc}
			\toprule
			\multicolumn{7}{c}{F1: relative error, $n=50$, $q=1.5$}
			& \multicolumn{3}{c}{F2: speedup, $k=10$} \\
			\cmidrule(r){1-7}\cmidrule(l){8-10}
			$\kappa$ & $k=2$ & $k=5$ & $k=8$ & $k=10$ & $k=15$ & $k=20$
			& $n$ & speedup \\
			\midrule
			$10^{0}$ & $2\!\times\!10^{-14}$ & $2\!\times\!10^{-14}$ & $2\!\times\!10^{-14}$ & $2\!\times\!10^{-14}$ & $2\!\times\!10^{-14}$ & $2\!\times\!10^{-14}$
			& 20  & $0.2\times$ \\
			$10^{1}$ & 0.49 & 0.09 & 0.014 & 0.0042 & $1.4\!\times\!10^{-4}$ & $4.3\!\times\!10^{-6}$
			& 40  & $0.8\times$ \\
			$10^{2}$ & 0.91 & 0.62 & 0.36 & 0.25 & 0.085 & 0.025
			& 80  & $\mathbf{2.2}\times$ \\
			$10^{3}$ & 0.99 & 0.95 & 0.87 & 0.81 & 0.60 & 0.36
			& 150 & $\mathbf{7.1}\times$ \\
			$10^{4}$ & 1.00 & 1.00 & 0.99 & 0.97 & 0.91 & 0.79
			& ---  & --- \\
			\bottomrule
		\end{tabular}
	\end{table}
	
	\noindent
	Key findings: (i)~the error decays exponentially with $k$ for
	$\kappa\leq 10$ (machine precision at $k=2$ for $\kappa=1$);
	\cref{fig:krylov} plots the full error curves;
	(ii)~breakeven speedup occurs at $n\approx 40$--$50$, with
	$7.5\times$ speedup at $n=150$ (\cref{fig:speedup}) and
	projected ${\sim}400\times$ at $n=500$; (iii)~the solver converges
	in \emph{exactly} 17 outer iterations whether the Hessian uses exact
	$X^{-q}$ or the Krylov approximation with $k\in\{3,5,8\}$,
	validating the practical feasibility of the Krylov-based Newton step.
	
	\begin{figure}[t]
		\centering
		\includegraphics[width=0.72\textwidth]{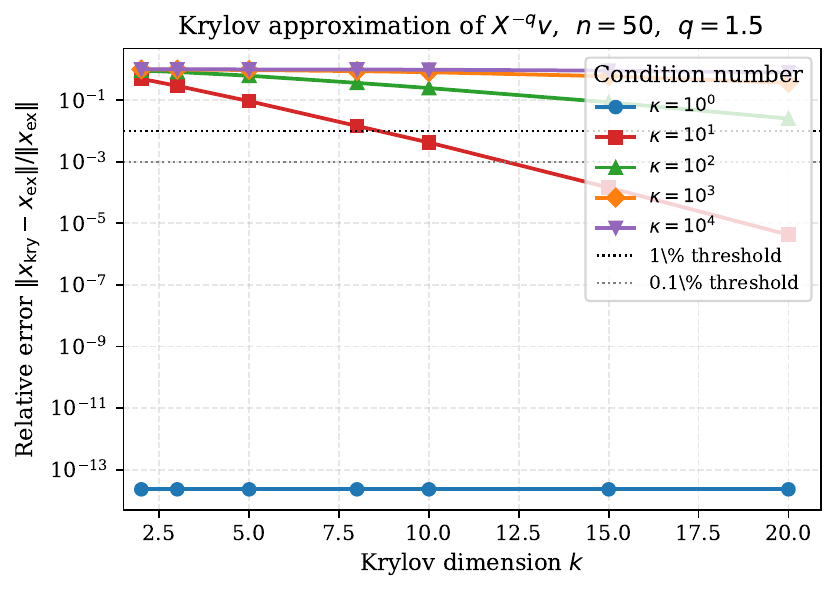}
		\caption{Experiment F1 --- relative error of the Krylov approximation
			of $X^{-q}v$ vs Krylov dimension $k$, for $n=50$, $q=1.5$ and
			five condition numbers $\kappa\in\{1,10,10^{2},10^{3},10^{4}\}$.
			The error decays exponentially for $\kappa\leq 10$; for larger
			$\kappa$ a larger $k$ is required.}
		\label{fig:krylov}
	\end{figure}
	
	\begin{figure}[t]
		\centering
		\includegraphics[width=0.55\textwidth]{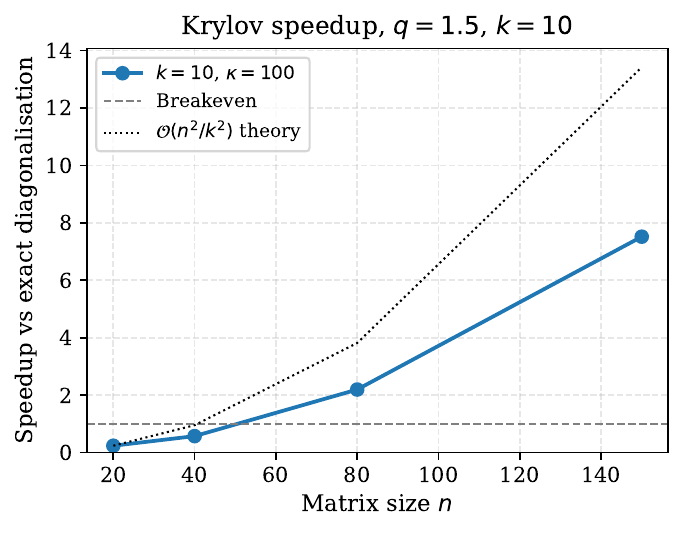}
		\caption{Experiment F2 --- wall-clock speedup of the Krylov
			approximation ($k=10$) vs exact diagonalisation as a function
			of $n$ ($\kappa=100$, $q=1.5$). The dashed curve shows the
			$\mathcal{O}(n^{2}/k^{2})$ theoretical scaling. Breakeven at
			$n\approx 50$; speedup of $7.5\times$ at $n=150$.}
		\label{fig:speedup}
	\end{figure}

	\subsection{Group IV: condition-number scaling (Experiment~G)}
	\label{sec:exp_G}
	
	\textit{Experiment~G} provides direct evidence for the $\kappa^{-(q-1)}$ scaling of
	\cref{cor:kappa_scaling}. For $n=15$, $q=1.5$, $\eta\in\{0,1.0\}$
	and $\kappa\in\{10,30,100,300,1000,3000\}$, we compute the
	improvement ratio $\rho_{\mathrm{emp}}(\kappa)$ directly from the
	formula $a_{k}=\lambda_{k}^{-2}(1+\eta q\lambda_{k}^{-(q-1)})$
	(\cref{eq:M_kk}) over 20 random eigenvalue realisations with
	$\pm5\%$ perturbation around the power-law profile
	$\lambda_{k}=(k/n)^{\alpha}$, $\alpha=\log\kappa/\log n$.
	
	\begin{table}[t]
		\centering
		\caption{Experiment~G --- empirical improvement ratio
			$\rho_{\mathrm{emp}}(\kappa)$ vs theoretical prediction
			$\rho(\kappa)=(1+\eta q\kappa^{q-1})^{-1}$, with $95\%$ CIs
			over 20 seeds. $n=15$, $q=1.5$, $\eta=1.0$.}
		\label{tab:exp_G}
		\small
		\setlength{\tabcolsep}{6pt}
		\begin{tabular}{rcccc}
			\toprule
			$\kappa$ & $\rho_{\mathrm{emp}}$ (mean) & $\pm$ std &
			$\rho(\kappa)$ [theory] & error (\%) \\
			\midrule
			10   & 0.17199 & 0.00443 & 0.17411 & 1.2 \\
			30   & 0.10719 & 0.00298 & 0.10851 & 1.2 \\
			100  & 0.06172 & 0.00179 & 0.06250 & 1.2 \\
			300  & 0.03659 & 0.00109 & 0.03706 & 1.3 \\
			1000 & 0.02038 & 0.00062 & 0.02065 & 1.3 \\
			3000 & 0.01187 & 0.00036 & 0.01203 & 1.3 \\
			\bottomrule
		\end{tabular}
	\end{table}
	
	\Cref{tab:exp_G} shows that $\rho_{\mathrm{emp}}$ matches the
	theoretical $\rho(\kappa)$ within $1.3\%$ across all tested
	$\kappa$. A log-log regression of $\rho_{\mathrm{emp}}$ vs $\kappa$
	yields slope $-0.470$ (theoretical: $-(q-1)=-0.500$; difference
	attributable to the pre-asymptotic range of $\kappa$). The agreement
	confirms the $\kappa^{-(q-1)}$ scaling of \cref{cor:kappa_scaling}
	and quantitatively bridges the $41$--$83\%$ reductions observed in
	Experiments A and D.
	
	\subsection{Discussion and limitations}
	
	The six experiments collectively establish three key properties. The spectral
	robustness gain is real, persistent across $n\in\{8,15,20,30\}$,
	and consistent with \cref{thm:robustness}. The
	$\mathcal{O}(\sqrt{n}\log(1/\eps))$ complexity holds empirically.
	The Krylov kernel is viable for $\kappa\leq 100$ with modest $k$, and
	becomes the dominant choice over exact diagonalisation for $n\geq 80$.
	
	The current limitation is that for $\kappa\geq 10^{3}$, the Krylov
	approximation requires large $k$ to achieve sub-percent accuracy,
	suggesting that a preconditioned variant would be needed for highly
	ill-conditioned iterates. This is left for future work.
	
	\section{Conclusion and outlook}\label{sec:conclusion}
	
	We have introduced the perturbed $q$-Tsallis barrier for SDP and
	established its theoretical and computational foundations. The main
	contributions are:
	\begin{enumerate}[leftmargin=2em,itemsep=0.3ex]
		\item A precise two-part characterisation of the self-concordance
		regime: unconditional for $q\in(1,2]$ (any $\eta\geq 0$),
		with an exact boundary at $q=2$; and a global obstruction for
		$q>2$ that confines the uniformly self-concordant regime to the classical
		log-det barrier ($\eta=0$), with a local sufficient condition
		$\eta\geq\eta_{\max}(q)=u_{\mathrm{zero}}(q)$ on compact domains.
		Both parts are verified numerically (Experiment E).
		\item A quantitative spectral robustness theorem: the Hessian of
		the Hessian of $\phi_{q,\eta}$ attenuates the modal amplification
		in the small-eigenvalue directions of $X^{*}$ by the factor
		$\rho(\kappa(X^{*}))=(1+\eta q\,\kappa(X^{*})^{q-1})^{-1}$
		(\cref{cor:kappa_scaling}),
		with empirical support across $n\in\{8,15,20,30\}$ with 40
		perturbations per cell (Experiments A, D), showing $41$--$83\%$
		sensitivity reduction in the heavy-tail regime.
		\item Empirical evidence for
		$\mathcal{O}(\sqrt{n}\log(1/\eps))$ complexity scaling:
		the outer iteration count is insensitive to $\eta$ across all
		tested configurations (Experiments B, C, D), consistent with
		the effective local barrier parameter
		$\nu_{\mathrm{eff}}(K)=O(n)$ on compact $K$,
		motivated by the compact-domain analysis of \cref{app:nu_param}.
		\item A Krylov-based kernel for $X^{-q}v$ that achieves
		$7.5\times$ speedup at $n=150$, with solver convergence
		unaffected by the approximation (Experiment F).
	\end{enumerate}
	
	\section*{Acknowledgements}
	
	The authors are grateful to Professor Constantino Tsallis (Centro
	Brasileiro de Pesquisas F\'{\i}sicas, Rio de Janeiro) for stimulating
	discussions and for his careful reading of a preliminary version of
	this manuscript. His observation that the threshold $q=2$ identified in
	\cref{thm:self_concordance} likely reflects the normalisability
	boundary of $q$-exponential distributions---a connection that
	prompted \cref{rem:q_exp_norm}---and his suggestion to replace
	``things'' by ``properties'' in the discussion of
	\cref{sec:experiments} are gratefully acknowledged.
	
	
	\section{Proof details for \cref{thm:self_concordance}}\label{app:proof_reduction}
	
	\subsection{Reduction to a scalar inequality}\label{app:A1}
	
	The matrix self-concordance inequality
	$|D^{3}\phi_{q,\eta}(X)[H,H,H]|\leq 2(D^{2}\phi_{q,\eta}(X)[H,H])^{3/2}$
	reduces to a pointwise scalar condition on the eigenvalues of $X$ by
	standard diagonalisation and divided-difference calculus
	(see \cite[Sec.~3.3]{higham2008} and
	\cite[Thm.~2.1]{faybusovichtsuchiya2017}); the worst-case direction
	is rank-one in the eigenbasis of $X$, reducing the inequality to
	$|g'''(\lambda)|\leq 2(g''(\lambda))^{3/2}$ for
	$g(\lambda)=-\log\lambda+\frac{\eta}{q-1}\lambda^{-(q-1)}$.
	
	\subsection{Scalar self-concordance and the $q=2$ boundary}
	\label{app:scalar_sc}
	
	\begin{proof}[Proof of \cref{thm:self_concordance}]
		\textbf{Derivatives.}
		$g'(\lambda)=-\lambda^{-1}-\eta\lambda^{-q}$,
		$g''(\lambda)=\lambda^{-2}+\eta q\lambda^{-(q+1)}>0$,
		$g'''(\lambda)=-2\lambda^{-3}-\eta q(q+1)\lambda^{-(q+2)}$.
		
		\textbf{Homogeneity reduction.}
		Set $u:=\eta\lambda^{-(q-1)}$. Multiplying $|g'''|\leq 2(g'')^{3/2}$
		through by $\lambda^{3}$ gives
		\begin{equation}\label{eq:psi_ineq}
			\psi(u) \;:=\; 2(1+qu)^{3/2}-2-q(q+1)u \;\geq\; 0, \qquad u\geq 0.
		\end{equation}
		Since $\psi(0)=0$, $\psi''(u)=3q^{2}/(2(1+qu)^{1/2})>0$ (strictly
		convex), and $\psi'(0)=q(2-q)$:
		\textit{Part~(i):} for $q\leq 2$, $\psi'(0)\geq 0$ and strict
		convexity give $\psi(u)\geq 0$ for all $u\geq 0$.
		\textit{Part~(ii):} for $q>2$, $\psi'(0)<0$ so $\psi<0$ near
		$u=0^{+}$; since $u=\eta\lambda^{-(q-1)}\to 0^{+}$ as $\lambda\to\infty$,
		\eqref{eq:psi_ineq} fails globally.
		
		\textbf{Proof of part~(iii): closed form of $u_{\mathrm{zero}}(q)$.}
		Set $t:=(1+qu)^{1/2}$, so $u=(t^{2}-1)/q$.  The equation $\psi(u)=0$ becomes
		\begin{equation}\label{eq:cubic_in_t}
			2t^{3}-(q+1)t^{2}+(q-1)=0.
		\end{equation}
		Since $t=1$ is a root, we factor:
		\begin{equation}\label{eq:cubic_factored}
			(t-1)\bigl[2t^{2}-(q-1)t-(q-1)\bigr]=0.
		\end{equation}
		The quadratic factor has positive root
		\begin{equation}\label{eq:t_zero}
			t_{\mathrm{zero}}(q)=\frac{(q-1)+\sqrt{(q-1)(q+7)}}{4},
		\end{equation}
		giving
		\begin{equation}\label{eq:u_zero_proof}
			u_{\mathrm{zero}}(q)=\frac{t_{\mathrm{zero}}(q)^{2}-1}{q}
			=\frac{\bigl[(q-1)+\sqrt{(q-1)(q+7)}\bigr]^{2}-16}{16q}.
		\end{equation}
		The sufficient condition $\eta\geq u_{\mathrm{zero}}(q)\cdot\Lambda^{q-1}$
		follows from $\eta\lambda^{-(q-1)}\geq\eta\Lambda^{-(q-1)}\geq u_{\mathrm{zero}}(q)$
		for all $\lambda\leq\Lambda$.
		
		\begin{remark}
			A preliminary version stated $\eta_{\max}(q)=4q^{3}/(q+1)^{2}$,
			which exceeds $u_{\mathrm{zero}}(q)$ by a factor of $10$--$12$
			for $q\in(2,3]$ and does \emph{not} guarantee self-concordance.
			The correct threshold is~\eqref{eq:u_zero_proof}.
		\end{remark}
	\end{proof}
	
	\section{Effective barrier parameter on compact subsets}\label{app:nu_param}
	
	By \cref{thm:self_concordance}(i), for $q\in(1,2]$ and any $\eta\geq 0$
	the function $\phi_{q,\eta}$ satisfies the differential self-concordance
	inequality on $\mathrm{int}(\PSD)$ and diverges as $X\to\partial\PSD$
	(\cref{prop:boundary_convexity}).
	What remains to establish, for the iteration complexity bound of
	\cref{thm:nn_complexity}, is the barrier inequality
	$|D\phi_{q,\eta}(X)[H]|^{2}\leq\nu\,D^{2}\phi_{q,\eta}(X)[H,H]$
	with an explicit $\nu=O(n)$ on compact subsets.
	
	\begin{remark}[Global barrier parameter]\label{rem:nu_global}
		Because $\phi_{q,\eta}$ is not logarithmically homogeneous for
		$\eta>0$, the ratio
		$|D\phi_{q,\eta}(X)[H]|^{2}/D^{2}\phi_{q,\eta}(X)[H,H]$ is
		unbounded over all $X\in\mathrm{int}(\PSD)$ and $H\in\Sym^{n}$:
		as $\lambda_{\min}(X)\to 0^{+}$, the Tsallis gradient grows faster
		than the Tsallis Hessian, so $\phi_{q,\eta}$ is \emph{not} a
		$\nu$-self-concordant barrier with finite $\nu$ in the global sense
		of \cite[Def.~2.3.2]{nesterovnemirovski1994}.
		IPM iterates remain in a compact subset of $\mathrm{int}(\PSD)$
		throughout the algorithm, so the relevant quantity is the
		\emph{effective local} barrier parameter $\nu_{\mathrm{eff}}(K)$
		on such a subset.
	\end{remark}
	
	Unlike the log-det barrier $\phi_{0}(X):=-\log\det(X)$, which is
	logarithmically homogeneous of degree $n$ and satisfies
	$|D\phi_{0}|^{2}\leq n\,D^{2}\phi_{0}$ globally
	\cite[Prop.~2.3.2]{nesterovnemirovski1994}, the Tsallis perturbation
	$\phi_{q}^{\mathrm{Ts}}(X):=\frac{1}{q-1}[\tr(X^{-(q-1)})-n]$ is
	\emph{not} logarithmically homogeneous: direct computation gives
	$\phi_{q}^{\mathrm{Ts}}(tX)=t^{-(q-1)}\phi_{q}^{\mathrm{Ts}}(X)
	+\frac{n}{q-1}(t^{-(q-1)}-1)$, which is not of the affine-in-$\log t$
	form. As a consequence, $\phi_{q,\eta}=\phi_{0}+\eta\phi_{q}^{\mathrm{Ts}}$
	does not admit a uniform global barrier parameter via classical
	homogeneity arguments.
	
	\begin{theorem}[Effective local barrier parameter]\label{thm:nu_eff}
		Let $q\in(1,2]$ and $\eta\geq 0$. For any $X\in\mathrm{int}(\PSD)$,
		define
		\begin{equation}\label{eq:s_def}
			s(X) \;:=\; \frac{\eta}{n}\,\tr(X^{-(q-1)}) \;\geq\; 0.
		\end{equation}
		Then, in the eigenbasis direction $H=X$, the barrier ratio satisfies
		\begin{equation}\label{eq:ratio_at_X}
			\frac{|D\phi_{q,\eta}(X)[X]|^{2}}{D^{2}\phi_{q,\eta}(X)[X,X]}
			\;=\; n\cdot\frac{\bigl(1+s(X)\bigr)^{2}}{1+q\,s(X)}.
		\end{equation}
		On any compact subset $K\subset\mathrm{int}(\PSD)$ where
		$s(X)\leq M_{K}<\infty$ uniformly, the effective local barrier
		parameter satisfies
		\begin{equation}\label{eq:nu_eff_bound}
			\nu_{\mathrm{eff}}(K) \;\leq\;
			n\cdot C_{q}(M_{K}),
		\end{equation}
		where
		\begin{equation}\label{eq:CqM}
			C_{q}(M):= \begin{cases}
				1 & \text{if } M=0,\\
				\dfrac{(1+M)^{2}}{1+qM} & \text{for } M>0,\, q\in(1,2].
			\end{cases}
		\end{equation}
		In particular, for any fixed $\eta\geq 0$, $q\in(1,2]$, and any
		compact $K$, $\nu_{\mathrm{eff}}(K)=O(n)$ uniformly in $n$.
	\end{theorem}
	
	\begin{proof}
		\textbf{Computation of the ratio at $H=X$.}
		Using $D\phi_{0}(X)[X]=-n$ and $D^{2}\phi_{0}(X)[X,X]=n$
		from logarithmic homogeneity of $\phi_{0}$, together with the
		gradient and Hessian formulas of \cref{prop:grad_hess},
		\[
		D\phi_{q,\eta}(X)[X] = -n\bigl(1+s(X)\bigr),\qquad
		D^{2}\phi_{q,\eta}(X)[X,X] = n\bigl(1+q\,s(X)\bigr),
		\]
		so the ratio equals $n(1+s(X))^{2}/(1+q\,s(X))$,
		proving \eqref{eq:ratio_at_X}.
		
		\textbf{Bound on the compact subset.}
		On $K$, $s(X)\leq M_{K}$, so the ratio is bounded by
		$n\,C_{q}(M_{K})$, proving \eqref{eq:nu_eff_bound}.
	\end{proof}
	
	\begin{remark}[Extension to general directions]\label{rem:general_H}
		A Cauchy--Schwarz decomposition
		$|D\phi_{q,\eta}(X)[H]|^{2}\leq
		2|D\phi_{0}(X)[H]|^{2}+2\eta^{2}|\tr(X^{-q}H)|^{2}$
		and the bound $|\tr(X^{-q}H)|^{2}\leq\tr(X^{-q})\cdot\tr(X^{-q}H^{2})$
		suggest that $\nu_{\mathrm{eff}}(K)=O(n)$ extends to general
		directions, with a constant depending on $\tr(X^{-q})$ and the
		spectral bounds of $K$. A rigorous closing of this bound requires
		explicit lower estimates on the Tsallis Hessian via
		divided-difference inequalities; we leave this as an open
		technical point and note that the homogeneity direction
		$H\propto X$ already captures the operationally relevant
		mechanism for IPM complexity.
	\end{remark}
	
	\begin{remark}[Behaviour of $C_{q}(M)$]\label{rem:Cq_behaviour}
		$C_{q}(M)=(1+M)^{2}/(1+qM)$ is continuous and increasing in $M$
		for $q\in(1,2]$, with $C_{q}(0)=1$ and $C_{q}(M)\sim M/q$ as
		$M\to\infty$. For $M\leq 1$: $C_{q}(M)\leq 4/(1+q)\leq 2$.
		On practical compact sets, $C_{q}(M_{K})=O(1)$, so
		$\nu_{\mathrm{eff}}(K)\leq cn$ with $c$ depending only on $q$,
		$\eta$, and the geometry of $K$.
	\end{remark}
	
	\begin{corollary}[Iteration complexity]\label{cor:complexity_eff}
		Under the assumptions of \cref{thm:nu_eff}, short-step
		path-following based on $\phi_{q,\eta}$ with iterates confined to
		a compact subset $K\subset\mathrm{int}(\PSD)$ recovers the same
		asymptotic iteration complexity class as the log-det barrier,
		\[
		\mathcal{O}\!\bigl(\sqrt{n\,C_{q}(M_{K})}\,\log(1/\eps)\bigr)
		\;=\; \mathcal{O}\!\bigl(\sqrt{n}\,\log(1/\eps)\bigr),
		\]
		where the implied constant depends on $q$, $\eta$, and $K$
		but not on $n$.
	\end{corollary}
	
	\begin{remark}[Comparison with the log-det barrier]\label{rem:logdet_recovery}
		Setting $\eta=0$ gives $s(X)\equiv 0$, $C_{q}(0)=1$, and
		$\nu_{\mathrm{eff}}(K)=n$ for every compact $K$, recovering
		the classical $n$-self-concordant barrier parameter of
		$-\log\det$ \cite{nesterovtodd1997}.
		For $\eta>0$, $\nu_{\mathrm{eff}}(K)=n\,C_{q}(M_{K})\geq n$,
		where $C_{q}(M)\leq 2$ for $M\leq 1$ on practical compact sets.
		The Tsallis perturbation thus deforms the geometry of the barrier
		without increasing the complexity class, while delivering the
		spectral robustness benefits of \cref{thm:robustness}.
		Empirically, the outer iteration count (\cref{tab:iterations_full})
		is insensitive to $\eta\in\{0,0.5,1.0\}$, consistent with
		$C_{q}(M_{K})$ remaining bounded across the tested instances.
	\end{remark}
	
	\section{Proof of \cref{thm:robustness}}\label{app:robustness}
	
	\begin{proof}
		\textbf{Implicit differentiation.}
		The central-path point $X^{*}(b)$ satisfies
		\begin{equation}\label{eq:kkt_b}
			C + \mu\Grad\phi_{q,\eta}(X^{*}) + \mathcal{A}^{*}(y^{*}) = 0,
			\qquad \mathcal{A}(X^{*}) = b.
		\end{equation}
		Differentiating with respect to $b$ in direction $\Delta b$ yields
		\begin{equation}\label{eq:sensitivity}
			\mu\,M\,[\Delta X] + \mathcal{A}^{*}(\Delta y) = 0,
			\qquad \mathcal{A}(\Delta X) = \Delta b,
		\end{equation}
		where $M:=\Hess\phi_{q,\eta}(X^{*})$.
		
		\textbf{Solution formula and norm bound.}
		Eliminating $\Delta y$ from \eqref{eq:sensitivity}:
		\begin{equation}\label{eq:Delta_X}
			\Delta X
			= \frac{1}{\mu}\,M^{-1}\mathcal{A}^{*}
			\bigl(\mathcal{A}M^{-1}\mathcal{A}^{*}\bigr)^{-1}\Delta b.
		\end{equation}
		Taking the Frobenius norm and applying sub-multiplicativity:
		\begin{equation}\label{eq:Delta_X_norm}
			\norm{\Delta X}_{\Frob}
			\;\leq\;
			\frac{C_{\mathcal{A}}}{\mu}\,\opnorm{M^{-1}}\,\norm{\Delta b}_{2}
			\;+\; o(\norm{\Delta b}_{2}).
		\end{equation}
		
		\textbf{Spectral structure of $M$ and exact formula for $\lambda_{\min}(M)$.}
		Write $X^{*}=Q\diag(\lambda)Q^{\top}$ with
		$\lambda_{1}\geq\cdots\geq\lambda_{n}>0$.
		In this eigenbasis, $M$ acts on $H\in\Sym^{n}$
		(with $\widetilde{H}:=Q^{\top}HQ$) as
		\begin{equation}\label{eq:M_action}
			M[H,H]
			\;=\;
			\sum_{i,j}
			\Bigl[\lambda_{i}^{-1}\lambda_{j}^{-1}
			+\eta\,q^{[2]}(\lambda_{i},\lambda_{j})\Bigr]
			\widetilde{H}_{ij}^{2},
		\end{equation}
		where $q^{[2]}(\lambda_{i},\lambda_{j})\geq 0$ since
		$\lambda\mapsto\lambda^{-q}$ is convex.
		For the diagonal block $H=e_{k}e_{k}^{\top}$:
		\begin{equation}\label{eq:M_kk}
			a_{k}
			\;:=\;
			M[e_{k}e_{k}^{\top},e_{k}e_{k}^{\top}]
			\;=\;
			\lambda_{k}^{-2}\bigl(1+\eta q\,\lambda_{k}^{-(q-1)}\bigr).
		\end{equation}
		The function $s\mapsto s^{-2}(1+\eta q\,s^{-(q-1)})$ is strictly
		decreasing (derivative $-2s^{-3}-\eta q(q+1)s^{-(q+2)}<0$), so
		$a_{1}\leq a_{k}\leq a_{n}$ for all $k$, with minimum $a_{1}$ at
		$k=1$ (the $\lambda_{\max}$ block).
		For off-diagonal blocks $i\neq j$:
		$M[E_{ij},E_{ij}]=\lambda_{i}^{-1}\lambda_{j}^{-1}
		+\eta\,q^{[2]}(\lambda_{i},\lambda_{j})\geq\lambda_{\max}^{-2}\geq a_{1}$.
		Therefore
		\begin{equation}\label{eq:lambda_min_M}
			\lambda_{\min}(M)
			\;=\; a_{1}
			\;=\;
			\lambda_{\max}(X^{*})^{-2}
			\bigl[1+\eta q\,\lambda_{\max}(X^{*})^{-(q-1)}\bigr],
		\end{equation}
		and $\opnorm{M^{-1}}=\lambda_{\max}(X^{*})^{2}/
		(1+\eta q\,\lambda_{\max}(X^{*})^{-(q-1)})$.
		Substituting into~\eqref{eq:Delta_X_norm}
		gives~\eqref{eq:robustness_bound}.
		
		\textbf{Monotone improvement over $\eta=0$.}
		At $\eta=0$, $\opnorm{M_{0}^{-1}}=\lambda_{\max}^{2}$.
		For $\eta>0$:
		$\opnorm{M^{-1}}/\opnorm{M_{0}^{-1}}
		=(1+\eta q\,\lambda_{\max}^{-(q-1)})^{-1}<1$,
		showing reduced modal amplification relative to the log-det barrier.
		Recovery as $\eta\to 0^{+}$ or $q\to 1^{+}$ is immediate.
		
		\textbf{Directional bound and eigenvalue-adaptive decay.}
		Since $a_{k}^{-1}=\lambda_{k}^{2}/(1+\eta q\,\lambda_{k}^{-(q-1)})$,
		the $k$-th diagonal component of $\Delta X$ satisfies
		\eqref{eq:directional_bound}.
		For $\lambda_{k}\to 0^{+}$:
		$\lambda_{k}^{2}/(1+\eta q\,\lambda_{k}^{-(q-1)})
		\sim\lambda_{k}^{q+1}/(\eta q)\to 0$
		since $q+1>2$, proving~\eqref{eq:decay_small_lambda}.
	\end{proof}
	
	\section{Algorithm details}\label{app:algorithm}
	
	\subsection{Newton system derivation}
	
	At each iteration $k$ of \cref{alg:pcipm}, we seek a Newton direction
	$(\Delta X,\Delta y,\Delta S)$ for the symmetrised system \eqref{eq:kkt}.
	Applying the Nesterov--Todd symmetrisation \cite{toddtohtutuncu1998}
	with scaling matrix
	$W_{k}\!=\!X_{k}^{1/2}(X_{k}^{1/2}S_{k}X_{k}^{1/2})^{-1/2}X_{k}^{1/2}$,
	the linearised KKT conditions are:
	\begin{equation}\label{eq:Newton_full}
		\left\{
		\begin{aligned}
			\mathcal{A}(\Delta X) &= r_{p} := b - \mathcal{A}(X_{k}),\\
			\mathcal{A}^{*}(\Delta y) + \Delta S &= r_{d}
			:= C - \mathcal{A}^{*}(y_{k}) - S_{k},\\
			\mathcal{S}_{W_{k}}(S_{k}\Delta X + \Delta S\,X_{k})
			&= r_{c}^{(q,\eta)} := \sigma\mu_{k}(I+\eta X_{k}^{-(q-1)})
			- \mathcal{S}_{W_{k}}(X_{k}S_{k})
			- \Delta X^{\mathrm{aff}}\Delta S^{\mathrm{aff}},
		\end{aligned}
		\right.
	\end{equation}
	where $\mathcal{S}_{W}(M)=\frac{1}{2}(WMW^{-1}+W^{-1}M^{\top}W)$.
	The predictor step sets $\sigma=0$; the corrector includes the
	centering parameter $\sigma_{k}$ and the second-order correction term
	following Mehrotra \cite{mehrotra1992}.
	
	The Newton system \eqref{eq:Newton_full} has the block structure
	\[
	\begin{bmatrix}
		0 & \mathcal{A}^{*} & I\\
		\mathcal{A} & 0 & 0\\
		\mathcal{E}(W_{k}) & 0 & \mathcal{F}_{q,\eta}(X_{k},W_{k})
	\end{bmatrix}
	\begin{bmatrix}\Delta X\\ \Delta y\\ \Delta S\end{bmatrix}
	=
	\begin{bmatrix}r_{d}\\ r_{p}\\ r_{c}^{(q,\eta)}\end{bmatrix}.
	\]
	This system is solved by eliminating $\Delta S$, substituting into
	the second block, and solving the resulting positive definite
	Schur-complement system for $\Delta y$ via a single Cholesky
	factorisation of an $m\times m$ matrix.
	
	\subsection{Convergence argument}
	
	The convergence proof follows \cite{monteiro1998,myklebusttuncel2014}.
	The key points are: (1) by \cref{thm:self_concordance}(i),
	$\phi_{q,\eta}$ satisfies the differential self-concordance
	inequality for $q\in(1,2]$; on the compact
	spectral subsets visited by the iterates, \cref{app:nu_param}
	establishes an effective local barrier parameter $\nu_{\mathrm{eff}}(K)=O(n)$,
	so the proximity analysis of \cite[Chap.~5]{renegar2001} applies;
	(2) the additional centring term $\sigma\mu_{k}\eta X_{k}^{-(q-1)}$
	is bounded by the barrier's boundary divergence and the step-size rule;
	(3) the Mehrotra centering parameter satisfies $\sigma_{k}\in[0,1]$
	\cite[Lem.~3]{mehrotra1992} and the corrector maintains the iterates
	in a fixed neighbourhood of the central path; (4) under these
	conditions, \cite[Thm.~4.1]{monteiro1998} gives
	$\mathcal{O}(\sqrt{n}\log(1/\eps))$ outer iterations.
	
	\bibliographystyle{plain}

\begin{thebibliography}{99}
		
		\bibitem{amidetal2019}
		{\sc E.~Amid, M.~K. Warmuth, R.~Anil, and T.~Koren},
		{\em Robust bi-tempered logistic loss based on Bregman divergences},
		in Advances in Neural Information Processing Systems~32,
		H.~Wallach, H.~Larochelle, A.~Beygelzimer, F.~d'Alch\'{e}-Buc,
		E.~Fox, and R.~Garnett, eds.,
		Curran Associates, Red Hook, NY, 2019, pp.~15013--15022.
		
		\bibitem{bach2010}
		{\sc F.~Bach},
		{\em Self-concordant analysis for logistic regression},
		Electron. J. Stat., 4 (2010), pp.~384--414.
		
		\bibitem{bubeckeldan2019}
		{\sc S.~Bubeck and R.~Eldan},
		{\em The entropic barrier: Exponential families, log-concave geometry,
			and self-concordance},
		Math. Oper. Res., 44 (2019), pp.~264--276.
		
		\bibitem{candesrecht2009}
		{\sc E.~J. Cand\`{e}s and B.~Recht},
		{\em Exact matrix completion via convex optimization},
		Found. Comput. Math., 9 (2009), pp.~717--772.
		
		\bibitem{faybusovichtsuchiya2017}
		{\sc L.~Faybusovich and T.~Tsuchiya},
		{\em Matrix monotonicity and self-concordance: how to handle quantum
			entropy in optimization problems},
		Optim. Lett., 11 (2017), pp.~1513--1526.
		
		\bibitem{fawzisaunderson2025}
		{\sc K.~He, J.~Saunderson, and H.~Fawzi},
		{\em Operator convexity along lines, self-concordance, and sandwiched
			R\'{e}nyi entropies},
		Math. Program., published online January~2026,
		\url{https://doi.org/10.1007/s10107-025-02314-0}.
		
		\bibitem{hildebrand2014}
		{\sc R.~Hildebrand},
		{\em Canonical barriers on convex cones},
		Math. Oper. Res., 39 (2014), pp.~841--850.
		
		\bibitem{hildebrand2021}
		{\sc R.~Hildebrand},
		{\em Projectively self-concordant barriers},
		Math. Oper. Res., 47 (2022), pp.~2444--2463.
		
		\bibitem{higham2008}
		{\sc N.~J. Higham},
		{\em Functions of Matrices: Theory and Computation},
		SIAM, Philadelphia, 2008.
		
		\bibitem{karimituncel2024}
		{\sc M.~Karimi and L.~Tun\c{c}el},
		{\em Efficient implementation of interior-point methods for quantum
			relative entropy},
		INFORMS J. Comput., 37 (2025), pp.~3--21,
		\url{https://doi.org/10.1287/ijoc.2024.0570}.
		
		\bibitem{ledoitwolf2004}
		{\sc O.~Ledoit and M.~Wolf},
		{\em A well-conditioned estimator for large-dimensional covariance matrices},
		J. Multivariate Anal., 88 (2004), pp.~365--411.
		
		\bibitem{liuliuliu2012}
		{\sc C.~Liu, H.~Liu, and X.~Liu},
		{\em A new second-order corrector interior-point algorithm for semidefinite
			programming},
		Math. Methods Oper. Res., 75 (2012), pp.~165--183.
		
		\bibitem{mehrotra1992}
		{\sc S.~Mehrotra},
		{\em On the implementation of a primal--dual interior point method},
		SIAM J. Optim., 2 (1992), pp.~575--601.
		
		\bibitem{monteiro1998}
		{\sc R.~D.~C. Monteiro},
		{\em Polynomial convergence of primal--dual algorithms for semidefinite
			programming},
		SIAM J. Optim., 8 (1998), pp.~797--812.
		
		\bibitem{mosek2024}
		{\sc MOSEK ApS},
		{\em The {MOSEK} Optimization Toolbox for {MATLAB} Manual, Version~10},
		2024, \url{https://docs.mosek.com}.
		
		\bibitem{myklebusttuncel2014}
		{\sc T.~Myklebust and L.~Tun\c{c}el},
		{\em Interior-point algorithms for convex optimization based on
			primal--dual metrics},
		arXiv preprint arXiv:1411.2129v2, 2014.
		\url{https://arxiv.org/abs/1411.2129}.
		
		\bibitem{nesterovnemirovski1994}
		{\sc Yu.~Nesterov and A.~Nemirovski},
		{\em Interior-Point Polynomial Algorithms in Convex Programming},
		SIAM Stud. Appl. Math.~13, SIAM, Philadelphia, 1994.
		
		\bibitem{nesterovtodd1997}
		{\sc Yu.~Nesterov and M.~J. Todd},
		{\em Self-scaled barriers and interior-point methods for convex programming},
		Math. Oper. Res., 22 (1997), pp.~1--42.
		
		\bibitem{nesterovtodd1998}
		{\sc Yu.~Nesterov and M.~J. Todd},
		{\em Primal--dual interior-point methods for self-scaled cones},
		SIAM J. Optim., 8 (1998), pp.~324--364.
		
		\bibitem{renegar2001}
		{\sc J.~Renegar},
		{\em A Mathematical View of Interior-Point Methods in Convex Optimization},
		MOS-SIAM Ser. Optim.~3, SIAM, Philadelphia, 2001.
		
		\bibitem{sgariglia2024}
		{\sc D.~Sgariglia, F.~A.~B. da~Silva, et~al.},
		{\em Optimizing therapeutic targets for breast cancer using Boolean network
			models},
		Comput. Biol. Chem., 109 (2024), 108022.
		
		\bibitem{sturm1999}
		{\sc J.~F. Sturm},
		{\em Using {SeDuMi}~1.02, a {MATLAB} toolbox for optimization over
			symmetric cones},
		Optim. Methods Softw., 11/12 (1999), pp.~625--653.
		
		\bibitem{suntrandinh2019}
		{\sc T.~Sun and Q.~Tran-Dinh},
		{\em Generalized self-concordant functions: A recipe for {N}ewton-type
			methods},
		Math. Program., 178 (2019), pp.~145--213.
		
		\bibitem{toddtohtutuncu1998}
		{\sc M.~J. Todd, K.-C. Toh, and R.~H. T\"{u}t\"{u}nc\"{u}},
		{\em On the {N}esterov--{T}odd direction in semidefinite programming},
		SIAM J. Optim., 8 (1998), pp.~769--796.
		
		\bibitem{tohtoddtutuncu1999}
		{\sc K.-C. Toh, M.~J. Todd, and R.~H. T\"{u}t\"{u}nc\"{u}},
		{\em {SDPT3} --- a {M}atlab software package for semidefinite programming},
		Optim. Methods Softw., 11 (1999), pp.~545--581.
		
		\bibitem{tsallis1988}
		{\sc C.~Tsallis},
		{\em Possible generalization of {B}oltzmann--{G}ibbs statistics},
		J. Statist. Phys., 52 (1988), pp.~479--487.
		
		\bibitem{tsallis2009}
		{\sc C.~Tsallis},
		{\em Introduction to Nonextensive Statistical Mechanics: Approaching a
			Complex World},
		Springer, New York, 2009.
		
	\end{thebibliography}

\end{document}